\documentclass[11pt, a4paper]{article}
\usepackage[latin1]{inputenc}
\usepackage[T1]{fontenc}
\usepackage[sorting=nyt]{biblatex}
\usepackage{amsmath}
\usepackage{amsfonts}
\usepackage{amssymb}
\usepackage{amsthm}
\usepackage[affil-it]{authblk}
\usepackage{enumitem}
\usepackage[normalem]{ulem}
\usepackage{dsfont}
\usepackage{float}
\usepackage{times}
\usepackage{bbm}
\usepackage{mathtools}

\usepackage[colorlinks,linkcolor=blue,citecolor=blue,urlcolor=blue]{hyperref}
\usepackage[capitalise]{cleveref}
\usepackage[left=3.5cm, right=3.5cm, bottom=3cm, top=3cm]{geometry}

\bibliography{reference.bib}

\newtheorem{theorem}{Theorem}[section]
\newtheorem{lemma}[theorem]{Lemma}
\newtheorem{proposition}[theorem]{Proposition}
\newtheorem{corollary}[theorem]{Corollary}

\theoremstyle{definition}
\newtheorem{defn}[theorem]{Definition}
\theoremstyle{remark}

\numberwithin{equation}{section}

\crefname{example}{Example}{Examples}

\newcommand{\N}{\ensuremath{\mathbb{N}}}

\newcommand{\R}{\ensuremath{\mathbb{R}}}

\renewcommand{\geq}{\geqslant}
\renewcommand{\leq}{\leqslant}

\begin{document}
%%%%%%%%%%%%%%%%%%%%%%%%%%%%%%%%%%%%%%%%%%%%%%%%%%%%%%%%%%%%%%%%%%%%%%%%%%%%%%%%%%%%
\title{McKean-Vlasov Equations with Positive Feedback through Elastic Stopping Times}
\author{Ben Hambly \footnote{Mathematical Institute, University of Oxford, United Kingdom, Email: hambly@maths.ox.ac.uk} \quad Julian Meier \footnote{Mathematical Institute, University of Oxford, United Kingdom, Email: julian.meier@maths.oc.ac.uk}} 
\date{\today} 

\maketitle
%%%%%%%%%%%%%%%%%%%%%%%%%%%%%%%%%%%%%%%%%%%%%%%%%%%%%%%%%%%%%%%%%%%%%%%%%%%%%%%%%%%%

\begin{abstract}
    We prove existence and uniqueness of physical and minimal solutions to McKean-Vlasov equations with positive feedback through elastic stopping times. We do this by establishing a relationship between this problem and a problem with absorbing stopping times. We show convergence of a particle system to the McKean-Vlasov equation. Moreover, we establish convergence of the elastic McKean-Vlasov problem to the problem with absorbing stopping times and to a reflecting Brownian motion as the elastic parameter goes to infinity or zero respectively.
\end{abstract}

\section{Introduction}
% The McKean-Vlasov equation with positive feedback through elastic stopping times that we consider is
% \begin{align} \label{ElasticProblem}
%     \begin{cases}
%     X_t = X_{0-} + B_t - \alpha \Lambda_t + L_t, \quad t \geq 0 \\
%     \Lambda_t = \mathbb{P}(\tau \leq t), \quad t \geq 0, \quad \text{where} \quad \tau = \inf \{t \geq 0: L_t \geq \xi \},
%     \end{cases}
% \end{align}
% where $\xi$ is an exponential random variable with parameter  $\kappa>0$, $X_{0-}$ is a random variable with law $\nu_{0-}, \alpha>0$ a constant, $B$ a standard Brownian motion and $L$ the reflection term we need to add to obtain a non-negative process. In the case of a continuous process, $L$ is the local time but as we will see the process $X$ may possess discontinuities and thus $L$ is the positivity preserving function in the corresponding Skorokhod problem. Feedback models with absorbing stopping times can be obtained by setting $\xi=0$.

McKean-Vlasov equations with positive feedback through hitting times have recently been used in the modelling of a range of phenomena, from the self-excitation of neurons \cite{Delarue2015a}, \cite{Delarue2015b}, through the modelling of default and systemic risk in banking networks, \cite{HamblyLedgerSojmark2019}, \cite{LedgerSojmark2020}, \cite{LedgerSojmark2021}, \cite{NadtochiySh2019}, \cite{NadtochiySh2020}, to a probabilistic representation for the supercooled Stefan problem \cite{delarue2020global}.
In these models, the feedback is generated through the dependence of the process on the distribution of its hitting times. These create singular interactions that can lead to jumps and cause complexity in the mathematical analysis such as non-uniqueness of solutions in general. 

Our aim in this paper is to consider a variant in which the feedback is provided through elastic stopping times. The McKean-Vlasov equation we consider is
\begin{align} \label{ElasticProblem}
    \begin{cases}
    X_t = X_{0-} + B_t - \alpha \Lambda_t + L_t, \quad t \geq 0 \\
    \Lambda_t = \mathbb{P}(\tau \leq t), \quad t \geq 0, \quad \text{where} \quad \tau = \inf \{t \geq 0: L_t \geq \xi \},
    \end{cases}
\end{align}
where $\xi$ is an exponential random variable with parameter  $\kappa>0$, $X_{0-}$ is a random variable with law $\nu_{0-}, \alpha>0$ a constant, $B$ a standard Brownian motion and $L$ the reflection term we need to add to obtain a non-negative process. In the case of a continuous process, $L$ is the local time but as we will see the process $X$ may possess discontinuities and thus $L$ is the positivity preserving function in the corresponding Skorokhod problem. Feedback models with absorbing stopping times can be obtained by setting $\xi=0$.

These problems are interesting from a mathematical point of view but can also be applied in several areas. For example, they can be used to model default cascades and systemic risk in the banking sector when default does not happen immediately. The elastic boundary can represent factors that can influence default such as negotiations over additional funds. In addition we will see that the elastic feedback model is closely connected to the supercooled Stefan problem with kinetic undercooling discussed in \cite{BakerShk2020}. 

The discontinuities that can occur as a result of the feedback lead to non-uniqueness. To resolve this problem, in the absorbing case, \cite{Delarue2015a} introduced the concept of a physical solution to such a McKean-Vlasov equation. Existence of such solutions has been established for several different feedback models \cite{Delarue2015a}, \cite{HamblyLedgerSojmark2019}, \cite{NadtochiySh2019}, while uniqueness remained an open problem until Delarue et al. \cite{delarue2020global} were able to use the connection to the super-cooled Stefan problem to prove uniqueness for a class of initial conditions with restricted oscillations. In \cite{Cuchiero2020}, the authors analyzed minimal solutions as an alternative natural solution concept. They were able to prove the existence of minimal solutions and show that minimal solutions are physical solutions.

In Section~2 we will establish existence and uniqueness of  physical and minimal solutions to the elastic feedback problem by exploiting a simple connection to a corresponding absorbing model. In Section~3 we show the connection to the Stefan problem with kinetic undercooling. We show weak convergence of a particle system to the McKean-Vlasov problem (\ref{ElasticProblem}) in Section~4. Finally, in Section~5, we prove convergence of $(\ref{ElasticProblem})$ to a reflecting Brownian motion and an absorbing feedback model when letting $\kappa$ go to zero or infinity, respectively. This is a probabilistic extension of the result in \cite{BakerShk2020} where convergence to the absorbing feedback model is shown under the assumption of continuous solutions.

\section{Physical and Minimal Solutions}

A solution to the problem (\ref{ElasticProblem}) is a deterministic cadlag function $t \mapsto \Lambda_t$, called the loss function, that starts at $0$, is increasing and satisfies the problem $(\ref{ElasticProblem})$.  We will define a physical and a minimal solution for this problem. We will then show how to exploit the results for absorbing equations to prove that (\ref{ElasticProblem}) has unique physical and minimal solutions and that these coincide.

The process $X$ is a reflected process that is the result of an application of the Skorokhod map $\Theta: D([0,\infty)) \to D([0,\infty))$ that maps a real-valued cadlag function to the corresponding non-negative reflected process (see \cite{Whitt} Chapter 13.5). We see that the process $X$ can be constructed by taking a process $Y$ defined as 
\begin{equation} \label{DefnYProcess}
    Y_t = X_{0-}+B_t - \alpha \Lambda_t
\end{equation}
and applying the Skorokhod map $\Theta$, i.e. $X = \Theta(Y)$. The term $L$ that is needed to keep $X$ nonnegative is given by
\begin{equation} \label{SkorMapL}
    L_t = (-\inf_{0 \leq s \leq t}Y_s ) \lor 0.
\end{equation}
We can use this to find the following expression for the elastic stopping time
\begin{equation*}
    \tau = \inf\{ t \geq 0 : L_t \geq \xi \} = \inf \{t \geq 0 : Y_t \leq - \xi \}
\end{equation*}
and the loss function $\Lambda_t$ becomes
\begin{equation*}
    \Lambda_t = \mathbb{P}(\tau \leq t) = \mathbb{P}(\inf_{0 \leq s \leq t} Y_s \leq - \xi).
\end{equation*}

A first result that is important in the elastic case uses a simple argument, similar to that proving \cite{HamblyLedgerSojmark2019} Theorem~1.1, to show that if the feedback parameter $\alpha$ is large enough, any solution to $(\ref{ElasticProblem})$ cannot be continuous for all times $t \geq 0$.

\begin{lemma}
Let $m_{0-} \coloneqq \int_0^{\infty} x \nu_{0-}(dx)$. If $\alpha > 2 (m_{0-} + \frac{1}{\kappa})$, then any solution $\Lambda$ to (\ref{ElasticProblem}) cannot be continuous for all times $t \geq 0$.
\end{lemma}

\begin{proof}
Assume $\alpha > 2 (m_{0-} + \frac{1}{\kappa})$ and $\Lambda$ is a continuous solution. We have that
\begin{equation*}
    0 \leq X_{t \wedge \tau} = X_{0-}+ B_{t \wedge \tau} - \alpha \Lambda_{t \wedge \tau} + L_{t \wedge \tau}.
\end{equation*}
After taking expectations and rearranging
\begin{equation*}
    m_{0-} \geq \alpha \mathbb{E}[\Lambda_{t \wedge \tau}] - \mathbb{E}[L_{t \wedge \tau}].
\end{equation*}
Taking the limit as $t \to \infty$ and noting that $\Lambda_{\infty}=1$, as $\inf_{t \geq 0}B_t = - \infty$, yields
\begin{align*}
    m_{0-} &\geq \alpha \int_0^{\infty} \Lambda_s d\Lambda_s - \mathbb{E}[\xi]  = \frac{1}{2}\alpha (\Lambda^2_{\infty}-\Lambda^2_0)- \frac{1}{\kappa} = \frac{1}{2} \alpha - \frac{1}{\kappa}.
\end{align*}
%where we have used the fact that 
Thus we get $2(m_{0-} + \frac{1}{\kappa}) \geq \alpha$, a contradiction.
\end{proof}

The fact that, as in the absorbing case, solutions in general are not continuous means we need to define a suitable solution concept. The two most natural ones to consider are that of a physical solution and a minimal solution. 

\begin{defn}[Minimal Solution]
A minimal solution $\underline{\Lambda}$ is a solution to $(\ref{ElasticProblem})$ such that for any other solution $\Lambda$ we have
\begin{equation*}
    \underline{\Lambda}_t \leq \Lambda_t, \quad t \geq 0.
\end{equation*}
\end{defn}
In the case of absorbing feedback models these solutions have been introduced and studied in detail in~\cite{Cuchiero2020}. 

Earlier the idea of a physical solution was introduced by Delarue et al.~\cite{Delarue2015a}. 
%In order to find the correct formulation of a physical solution for the elastic model it is 
%helpful to recall the concept in the case of an absorbing model. In \cite{Delarue2015a}
A physical solution, in the absorbing case, is defined as a solution $\Lambda$ such that the condition
\begin{equation} \label{absorbingPhysial}
    \Delta \Lambda_t = \inf \{x \geq 0: \mathbb{P}(t<\tau^a, 0 < X_{0-} + B_t - \alpha \Lambda_{t-}<\alpha x) < x \}
\end{equation}
is satisfied, where $\tau^{a}=\inf\{t\geq 0: X_t\leq 0\}$ denotes the absorbing stopping time.

Recalling the definitions of the elastic stopping time in $(\ref{ElasticProblem})$ and the map $L$ from (\ref{SkorMapL}) an equivalent condition for a solution $\Lambda$ to the elastic feedback model can be defined.

\begin{defn}[Physical Solution]
A physical solution $\Lambda$ to $(\ref{ElasticProblem})$ is a solution such that the condition
\begin{equation} \label{elasticPhysical}
    \Delta \Lambda_t = \inf \{x \geq 0: \mathbb{P} (t < \tau, 0 < X_{0-}+ \xi +B_t - \alpha \Lambda_{t-}<\alpha x) < x  \}
\end{equation}
holds for all $t \geq 0$.
\end{defn}

We will establish the existence and uniqueness of a physical solution for $(\ref{ElasticProblem})$. The key idea is that we can relate the elastic feedback model to an absorbing feedback model for which these are known results by \cite{delarue2020global}. 
%This connection is the content of the following lemma.

\begin{lemma} \label{ElasticAbsorbingLemma}
Let $\Lambda$ be an increasing cadlag function with $\Lambda_0=0$. Then $\Lambda$ is a solution to the McKean-Vlasov problem (\ref{ElasticProblem}) if and only if it is a solution to the absorbing feedback model 
\begin{align} \label{AbsorbingProblem}
    \begin{cases}
        Y_t = Y_{0-}+B_t-\alpha \Lambda^{a}_t, \quad t \geq 0 \\
         \Lambda_t^{a} = \mathbb{P}(\tau^a \leq t), \quad t \geq 0, \quad \text{where} \quad \tau^a = \inf \{t \geq 0 : Y_t \leq 0 \}
    \end{cases}
\end{align}
with initial condition $Y_{0-} \coloneqq X_{0-}+\xi$.
\end{lemma}

\begin{proof}
Let $\Lambda$ be a solution to the problem (\ref{ElasticProblem}). We can rewrite this using the definition of the stopping time $\tau$ and the Skorokhod map $\Theta$ in the following way
\begin{align*}
    \Lambda_t =\mathbb{P}(\tau \leq t) = \mathbb{P}(L_t \geq \xi) &= \mathbb{P}(\inf_{s \leq t}(X_{0-}+B_s-\alpha \Lambda_s) \leq - \xi) \\
    &= \mathbb{P}(\inf_{s \leq t}(\xi + X_{0-}+B_s-\alpha \Lambda_s) \leq 0) \\
    &= \Lambda^a_t.
\end{align*}
Therefore, a solution to the elastic feedback model (\ref{ElasticProblem}) is also a solution to the absorbing model $(\ref{AbsorbingProblem})$. Conversely, any solution to $(\ref{AbsorbingProblem})$ solves the elastic model.
\end{proof}

\begin{theorem} \label{ThmPhysicalSoltn}
Let $X_{0-}$ be a random variable with bounded density $f$ such that $f$ changes monotonicity at most finitely often on any compact interval. Then there exists a unique physical solution to the elastic feedback model (\ref{ElasticProblem}).
\end{theorem}

\begin{proof}
By Lemma \ref{ElasticAbsorbingLemma} we know that an increasing cadlag process $\Lambda$ solves the McKean-Vlasov problem (\ref{ElasticProblem}) if and only if it solves the problem (\ref{AbsorbingProblem}). Additionally, considering the definition of a physical solution to the absorbing model (\ref{absorbingPhysial}) and to the elastic model (\ref{elasticPhysical}), we see that the same holds for physical solutions. Hence a map $\Lambda$ is a physical solution to the elastic feedback model if and only if it is a physical solution to the absorbing feedback model (\ref{AbsorbingProblem}). The random variable $Y_{0-}$ with the law of $X_{0-}+\xi$, has a density $g$ which inherits the properties from $f$, i.e. $g$ is bounded and changes monotonicity at most finitely often on compact intervals. Moreover, since $\xi$ is an exponential random variable we obtain from the convolution formula that $g(0)=0$. We can apply the results from \cite{delarue2020global} for absorbing feedback models and conclude that the elastic feedback model $(\ref{ElasticProblem})$ has a unique physical solution. 
\end{proof}

Similarly, Lemma  \ref{ElasticAbsorbingLemma} can be combined with the ideas from \cite{Cuchiero2020} about minimal solutions.
\begin{theorem}
The elastic feedback model has a unique minimal solution. Moreover, if the conditions of Theorem \ref{ThmPhysicalSoltn} hold, then the unique minimal solution and the unique physical solution coincide.
\end{theorem}

\begin{proof}
Using the reformulation in terms of an absorbing feedback problem from Lemma \ref{ElasticAbsorbingLemma}, by \cite{Cuchiero2020} there exists a unique minimal solution and the minimal solution is a physical solution. If we additionally have that the density of $X_{0-}$ is bounded and changes monotonicity at most finitely often on compact intervals, Theorem \ref{ThmPhysicalSoltn} applies and the physical solution is unique.
\end{proof}

\section{PDE Formulation and Connection with the Kinetic Undercooled Stefan Problem}

In \cite{BakerShk2020} the authors study a free boundary problem called the supercooled Stefan problem with kinetic undercooling as an approximation to the supercooled Stefan problem. They work in the regime without jumps and use a probabilistic representation of solutions based on a Feynmac-Kac formula to derive their results. In this section we show, under the assumption of continuous solutions, that the elastic feedback model is a probabilistic formulation of the Stefan problem with kinetic undercooling. Thus, the results in Section 2 on existence and uniqueness, as well as the approximation results discussed in Section 5, are a probabilistic generalisation of the results in \cite{BakerShk2020}.
We define a measure $\nu$ by its actions on  test functions $\phi \in C^2$ with $\kappa \phi(0) = \partial_x \phi(0)$ as $\nu_t(\phi)  \coloneqq \mathbb{E}[\phi(X_t)\mathbbm{1}_{t < \tau}]$. To obtain the dynamics we calculate $\phi(X_{t \wedge \tau})$ by applying Ito's formula 
\begin{align*}
    d\phi(X_{t \wedge \tau}) &= \partial_x \phi(X_{t \wedge \tau})dB_t - \alpha \partial_x \phi(X_{t \wedge \tau})d\Lambda_t + \partial_x \phi(X_{t \wedge \tau})dL_t + \frac{1}{2} \partial_{xx} \phi(X_{t \wedge \tau})dt.
\end{align*}
As $X_{\tau} = 0$, we have
\begin{equation*}
    \phi(X_{t \wedge \tau}) = \phi(X_t) \mathbbm{1}_{t < \tau} + \phi(0) \mathbbm{1}_{\tau \leq t}.
\end{equation*}
Using this fact and taking expectations we have the measure dynamics
\begin{align*}
    d\nu_t(\phi) &= \frac{1}{2} \nu_t(\partial_{xx} \phi)dt - \alpha \nu_t(\partial_x \phi)d\Lambda_t  + \kappa \phi(0) \mathbb{E}[L_{t \wedge \tau}] - \phi(0) \mathbb{P}(\tau \leq t).  
\end{align*}
We focus on the last two terms and compute 
\begin{align*}
    \kappa \mathbb{E}[L_{t \wedge \tau}] - \mathbb{P}(\tau \leq t) & =  \kappa \mathbb{E}[L_t \mathbbm{1}_{t < \tau}]  + \kappa \mathbb{E}[\xi \mathbbm{1}_{t \geq \tau}] - \mathbb{P}(\tau \leq t) \\
    &= \kappa \mathbb{E}[L_t e^{-\kappa L_t}]- \kappa\mathbb{E}[ L_t e^{- \kappa L_t}] + \mathbb{P}(\tau \leq t) -\mathbb{P}(\tau \leq t) = 0.
\end{align*}
Thus we obtain the weak form of the nonlinear PDE
\begin{align*}
\begin{cases}
    \nu_t(\phi) = \nu_0(\phi) + \frac{1}{2} \int_0^t \nu_s(\partial_{xx} \phi)ds - \alpha \int_0^t \nu_s(\partial_x\phi)d\Lambda_s, \\
    \Lambda_t = 1- \int_0^{\infty} \nu_t(dx).
\end{cases}
\end{align*} 

By removing the continuous drift using Girsanov's theorem and applying the Radon-Nikodym Theorem it can be shown that the measure $\nu_t$ possesses a density $V(t,x)$. We can use integration by parts to derive the PDE governing the evolution of the density
\begin{equation*}
    \begin{cases}
        \partial_t V(t,x) = \frac{1}{2} \partial_{xx} V(t,x) + \alpha \Dot{\Lambda}_t \partial_x V(t,x), \\
        \frac{1}{2} \partial_x V(t,0) = \left(\frac{\kappa}{2} - \alpha \Dot{\Lambda}_t \right) V(t,0). 
    \end{cases}
\end{equation*}
Now we reparametrize and set $\beta=\frac{2}{\alpha}$, $\mathcal{L} = \frac{\alpha}{2}\Lambda$,   and for an $\varepsilon>0$ set $\kappa=\beta/\varepsilon$. Finally, writing $u(t,x)=V(t,x-\mathcal{L}_t)$, we have that $u$ solves the free boundary problem
\begin{equation*}
    \begin{cases}
        \partial_t u = \frac{1}{2} \partial_{xx} u, \quad \text{on } \Gamma \coloneqq \{(t,x) = [0,\infty)^2 : x \geq \mathcal{L}_t \}, \\
        u(t,\mathcal{L}_t) = \varepsilon \Dot{\mathcal{L}}_t, \quad t \geq 0, \\
        (\beta-2 u(t,\mathcal{L}_t))\Dot{\mathcal{L}}_t = \partial_x u(t,\mathcal{L})_t, \quad t \geq 0, \quad \mathcal{L}_0 = 0.
    \end{cases}
\end{equation*}
This is the Stefan problem with kinetic undercooling as discussed in  \cite{BakerShk2020}.

\section{Particle System Approximation}

We consider a system of particles with elastic feedback given by
\begin{align} \label{particleSystem}
    \begin{cases}
        X_t^{i,N} = X_{0-}^i + B_t^i - \alpha \Lambda_t^N + L_t^{i,N}, \quad t \geq 0 \\
        \Lambda_t^N = \frac{1}{N} \sum_{i=1}^N \mathbbm{1}_{\tau_{i,N} \leq t}, \quad t \geq 0, \quad \text{where} \quad \tau_{i,N} = \inf\{t \geq 0, L_t^{i,N} \geq \xi^i \},
    \end{cases}
\end{align}
where the parameters are as discussed in Section~1 and the random variables $\xi_i$ are i.i.d.
We can use the ideas from Section 2 and formulate this as a system of particles with absorbing boundary and initial conditions $X_{0-}^i+\xi^i$. This way, we see that the theory for absorbing particle systems also holds for the elastic ones. As for our McKean-Vlasov problem, minimal and physical solutions can be defined for this particle system. In \cite{Cuchiero2020} existence of a minimal solution is shown. Moreover, the authors establish that the physical solution is equal to the minimal solution. 

The goal of this section is to establish weak convergence, on the Skorokhod space with the M1-topology \cite{Whitt}, of the system (\ref{particleSystem}) to a solution to the McKean-Vlasov equation (\ref{ElasticProblem}). After verifying tightness in the M1-topology we show that physical solutions to the particle system converge to physical solutions of (\ref{ElasticProblem}). We can then refer to the uniqueness results in Theorem \ref{ThmPhysicalSoltn} to deduce weak convergence. 
The main idea is again to use the structure of the processes $X_t^{i,N}$ in terms of the Skorokhod map $\Theta$, that is
\begin{equation*}
    X_t^{i,N} = \Theta(Y_t^{i,N}) =\Theta(X_{0-}^i+B_t^i- \alpha \Lambda_t^N)
\end{equation*}
and exploit the fact that the Skorokhod map is continuous on the Skorokhod space with the M1-topology (see \cite{Whitt} Theorem 13.5.1).
\subsection{Tightness}
We closely follow the approach in \cite{Cuchiero2020}. We consider the loss processes $\Lambda^N$ to be elements of the space $M$ defined by
\begin{equation} \label{spaceM}
    M \coloneqq \{l: \bar{\R} \to [0,1] \vert l \text{ cadlag and increasing, } l_{0-}=0, l_{\infty}=1 \}.
\end{equation}
Writing $\mathcal{P}(E)$ for the space of probability measures on a Polish space $E$, the space $M$ can be identified with distribution functions of measures in $\mathcal{P}([0,\infty])$. If we equip it with the topology of weak convergence it becomes a compact Polish space, because the topology is metrizable (for example using the Levy metric).
Define the spaces $\Bar{E}$ and $\hat{E}$ as
\begin{equation*}
    \Bar{E} \coloneqq C([0,\infty)) \times M, \quad \hat{E} \coloneqq \Bar{E} \times \R.
\end{equation*}
    
\begin{theorem} \label{MinusOneExtension}
Endowed with the product topology induced by compact convergence on $C([0,\infty))$, and the Levy-metric on M, the space $\Bar{E}$ and then also $\hat{E}$ are Polish. For $w \in C([0,\infty))$ and $l \in M$, define
\begin{align*}
    \hat{w}_t \coloneqq 
    \begin{cases}
        w_0 \quad t \in [-1,0) \\
        w_t \quad t \in [0,\infty)
    \end{cases}
    \quad
    \check{l}_t \coloneqq
    \begin{cases}
        0 \quad t \in [-1,0) \\
        l_t \quad t \in [0,\infty).
    \end{cases}
\end{align*}
For any $\alpha \in \R$ the embedding $\iota_{\alpha}: \hat{E} \to (D([-1,\infty)) \times \R)$ defined via
\begin{equation*}
    \iota_{\alpha}(w,l,z) = (\hat{w}-\alpha \check{l},z)
\end{equation*}
is continuous.
\end{theorem}

\begin{proof}
The space $\hat{E}$ is Polish as the product space of Polish spaces. The fact that the map $(w,l)\mapsto (\hat{w}-\alpha \check{l})$ is a continuous map from $\Bar{E}$ to $D([-1,\infty))$ then follows as in the proof of \cite{Cuchiero2020} Theorem 4.2. As a consequence $\iota_{\alpha}$ is a continuous map.
\end{proof}

%\noindent
%Note that through the extension to $[-1,\infty)$ the Skorokhod map adds a local time term starting at $-1$. This might be problematic, but since we are restricting ourselves to initial conditions supported on the positive half-line and the extension of the path only involves $w_0$ the local time is $0$ on the interval $[-1,0]$.

\begin{theorem} \label{TightnessTheorem}
Let $(X^N)_{N \in \N}$ be a sequence of stochastic processes with paths in $D([0,\infty))$. Suppose that $X^N$ is of the form

\begin{equation*}
    X^N = \Theta(Y^N) =\Theta(Z^N - \alpha_N \Lambda^n)  
\end{equation*}
\noindent
with $\alpha_N$ a real number, $Z^N$ continuous and $\Lambda^N \in M$. Suppose that $(Z^N)_{N \in \N}$ is tight on $C([0,T])$ for each $T>0$ and $(\alpha_N)_{N \in \N}$ is tight on $\R$. Furthermore, let $(\xi^N)_{N \in \N}$ be a tight sequence of random variables on $\R$. Following the notation from Theorem \ref{MinusOneExtension} we denote by $\hat{Y}^N$ the extension of $Y^N$ to $D([-1,\infty))$ and set $\hat{X}^N =\Theta(\hat{Y}^N)$. 
Then the random variables $((Z^N,\Lambda^N,\xi^N))_{N \in \N}$ are tight on $\hat{E}$, the random variables $((\hat{Y}^N,\xi^N))_{N \in \N}$ are tight on $D([-1,\infty)) \times \R$  and the random variables $(\hat{X}^N)_N$ are tight on $D([-1,\infty))$.
\end{theorem}
\begin{proof}
By \cite{Cuchiero2020} Theorem 4.4 the random variables $((Z^N,\Lambda^N))_{N \in \N}$ are tight on $\bar{E}$ and the random variables $(Y^N)_{N \in \N}$ are tight on $D([-1,\infty)$. Since $(\xi^N)_{N \in \N}$ are tight on $\R$ by assumption it follows that $((Z^N,\Lambda^N,\xi^N))_{N \in \N}$ are tight on $\hat{E}$ and  $((\Hat{Y}^N,\xi^N))_{N \in \N}$ are tight on $D([-1,\infty)) \times \R$. By continuity of the Skorokhod map it then follows that the random variables $(\hat{X}^N)_{N \in \N}$ are tight on $D([-1,\infty))$.
\end{proof}
In the rest of this article we will no longer distinguish between processes on $D([0,\infty))$ and their extensions to $D([-1,\infty))$.
Theorem \ref{TightnessTheorem} yields tightness of the processes we consider, but we need tightness for the corresponding empirical measures. By the following result of \cite{Snitzman}, if we have exchangeability, then this tightness follows from the tightness of a sequence of particles. 
% This can be obtained from the following result of Snitzman \cite{Snitzman} for which we first need to introduce the notion of exchangeability.

% Set $X^N \coloneqq (X^{1,N}, \ldots, X^{N,N})$.
% \begin{defn}
% $X^N$ is called $N$-exchangeable, if
% \begin{equation*}
%     \mbox{law}(X^N) = \mbox{law} ((X^{\sigma(1),N},X^{\sigma(2),N}, \ldots, X^{\sigma(N),N}))
% \end{equation*}
% for any permutation $\sigma$ of $\{1, \ldots, N\}$.
% \end{defn}

\begin{proposition} [\cite{Snitzman}, Proposition 2.2.]\label{SnitzmanProp}
Let $E$ be a Polish space and let $X^N\coloneqq (X^{1,N}, \ldots, X^{N,N})$ be $N$-exchangeable on $E^N$ for every $N \in \N$, then the $\mathcal{P}(E)$-valued random variables $(\mu_N)_{N \in \N}$ given by
\begin{equation*}
    \mu_N = \frac{1}{N} \sum_{i=1}^N \delta_{X^{i,N}}
\end{equation*}
are tight if and only if $(X^{1,N})_{N \in \N}$ is tight on $E$.
\end{proposition}

%\begin{proof}
%Snitzman \cite{Snitzman}, Proposition 2.2.
%\end{proof}

\begin{corollary} \label{empiricalMeasureTightness}
Suppose that $X^N$ has dynamics
\begin{equation*}
    X_t^{i,N} = \Theta(Y^{i,N}) = \Theta(X_{0-}^{i,N} + B_t^i - \alpha \Lambda_t^N)
\end{equation*}
where $\alpha >0$, $(X_{0-}^N)_{N \in \N}$ and $(\xi^{i,N})_{N \in N}$ are $N$-exchangeable random vectors, $(B_i)_{i \in \N}$ independent Brownian motions, and $\Lambda^N \in M$. If $(X_{0-}^{1,N})_{N \in \N}$ and $(\xi^{1,N})_{N \in \N}$ are tight on $\R$, then the empirical measures
\begin{equation} \label{EmpiricalMeasuresDefn}
    \mu_N^{\Theta}=\frac{1}{N} \sum_{i=1}^N \delta_{X^{i,N}}  \quad
    \mu_N = \frac{1}{N} \sum_{i=1}^N \delta_{(Y^{i,N},\xi^{i,N})}, \quad \text{ and } \quad \zeta_N = \frac{1}{N} \sum_{i=1}^N \delta_{(X_{0-}^{i,N}+B^i,\Lambda^N,\xi^{i,N})}
\end{equation}
are tight on $\mathcal{P}(D([-1,\infty)))$, $\mathcal{P}(D([-1,\infty)) \times \R)$, and $\mathcal{P}(\hat{E})$ respectively.
\end{corollary}

\begin{proof}
Since $(X_{0-}^{1,N})_{N\in \N}$ is tight on $\R$ it easily follows that $(X_{0-}^{1,N}+B^1)_{N \in \N}$ is tight on $C[0,T]$ for all $T>0$. Thus, the conditions of Theorem \ref{TightnessTheorem} are satisfied and we can combine this with Proposition \ref{SnitzmanProp} to finish the proof.
\end{proof}

We will write $f(\mu)$ for the pushforward of a measure $\mu$ by a map $f$.

\begin{corollary} 
Let $(\mu^{\Theta}_N)_{N \in \N}$ and $(\mu_N)_{N \in \N}$ be as in the Corollary \ref{empiricalMeasureTightness}. Then there are $\mathcal{P}(\hat{E})$-valued random variables $\zeta, \zeta_N$ such that, at least for a subsequence, $\mbox{law}(\mu_N)= \mbox{law}(\iota_{\alpha}(\zeta_N)), \zeta_N \to \zeta$ and $\iota_{\alpha}(\zeta_N) \to \iota_{\alpha}(\zeta)$ a.s. It holds that $\mbox{law}(\mu_N^{\Theta})= \mbox{law}(\Theta(\mu_N))$ and $\Theta(\mu_N) \to \Theta(\iota_{\alpha}(\zeta))$ a.s. Moreover,
\begin{equation*}
    \mbox{law}(\zeta_N) = \mbox{law} \left( \frac{1}{N} \sum_{i=1}^N\delta_{(X_{0-}^{i,N}+B^i, \Lambda^N,\xi^{i,N})}\right).
\end{equation*}
\end{corollary}
\begin{proof}
The result follows by combining Corollary \ref{empiricalMeasureTightness}, the Skorokhod representation theorem and the continuity of the maps involved.
\end{proof}

\subsection{Convergence of Solutions}

We consider the setting of Corollary \ref{empiricalMeasureTightness}. The three empirical measures given in (\ref{EmpiricalMeasuresDefn}) have the following relations
\begin{equation*}
    \mbox{law}(\mu_N) = \mbox{law}(\iota_{\alpha}(\zeta_N)); \quad \mbox{law}(\mu_N^{\Theta})= \mbox{law}(\Theta(\mu_N)).
\end{equation*}

\begin{defn}
For $t \in \R$, $x \in D([-1,\infty))$ and $z \in \R$ define the following path functionals
\begin{equation*}
    \tau_0(x,z) \coloneqq \inf \{s \geq 0 : x_s \leq - z \} \quad \text{and} \quad \lambda_t(x,z) \coloneqq \mathbbm{1}_{\tau_0(x,z) \leq t}
\end{equation*}
\end{defn}

\begin{lemma} \label{IndicatorContinuity}
Assume that $(x^n,z^n) \to (x,z)$ in $(D([-1,\infty)) \times \R)$ with $(x^n,z^n),(x,z) \in \iota_{\alpha}(\hat{E})$ and $x$ satisfies the crossing property 
\begin{equation*}
    \inf_{0 \leq s \leq h} (x_{\tau(x,z)+s}-x_{\tau(x,z)})<0, \quad h>0.
\end{equation*}
Then we have that
\begin{equation*}
    \lim_{n \to \infty} \lambda_t(x^n,z^n) = \lambda_t(x,z)
\end{equation*}
for all $t \in [0,\infty)$ in a co-countable set.
\end{lemma}

\begin{proof}
Define the map 
\begin{equation*}
    \Xi: (D([-1,\infty)) \times \R) \to D([-1,\infty)); (x,z) \mapsto x+z.
\end{equation*}
This map is continuous with respect to the $M1$-topology and the function $\Xi(x,z)$ satisfies the crossing property. We have
\begin{equation*}
    \lambda_t(x,z) = \lambda_t(\Xi(x,z),0).
\end{equation*}
The convergence then follows as in the proof of \cite{Cuchiero2020} Lemma 5.4.
\end{proof}

\begin{lemma} \label{LossConvergenceLemma}
Assume that $(\zeta^n)_{n \in \N}$ is a convergent sequence of probability measures on $\hat{E}$ with limit $\zeta$. Define $\mu^n \coloneqq \iota_{\alpha}(\zeta^n)$ and $\mu \coloneqq \iota_{\alpha}(\zeta)$. If $\mu$-a.e path satisfies the crossing property
\begin{equation} \label{crossingPropertyMeasure}
    \mu(\{(x,z) \in D([-1,\infty)) \times \R: \inf_{0 \leq s \leq h}(x_{\tau_0(x,z)+s}-x_{\tau_0(x,z)})=0 \})=0, \quad h >0,
\end{equation}
then $\lim_{n \to \infty} \langle \mu^n, \lambda \rangle = \langle \mu, \lambda \rangle$ in $M$.
\end{lemma}

\begin{proof}
The map $\iota_{\alpha}$ is continuous. As a result, the convergence  $\mu^n \to \mu$ in $\mathcal{P}(D[-1,\infty)\times \R)$ follows. We set $l_t^n \coloneqq \langle \mu^n, \lambda_t \rangle$ and note that for all $n \in \N$ the map $l^n$ is in $M$. As mentioned at the beginning of the section, the space $M$ is compact. Thus,  we have the existence of a limit point $l \in M$. We need to prove that $l = \langle \mu, \lambda \rangle$. This can be done following the arguments in the proof of \cite{Cuchiero2020} Lemma 5.5 and replacing their path functionals by the paths functional $\tau_0(x,z)$ and $\lambda_t(x,z)$ considered here.
\end{proof}

The following two lemmas are important tools in the proof of convergence to the McKean-Vlasov problem. Their proofs follow from the argument used in \cite{Cuchiero2020}, replacing the path functionals
with the functionals $\tau_0(x,z)$ and $\lambda_t(x,z)$, and extending the measures involved to the spaces $\hat{E}$ and $D([-1,\infty)) \times \R$.

\begin{lemma} \label{BrownianMotionLemma}
For almost every realization $(\omega,z)$, if $\mbox{law}((W,\Lambda,\xi))= \zeta(\omega,z)$, then $W-W_0$ is a Brownian motion with respect to the filtration generated by $(W,\Lambda)$. In particular, $W-W_0$ is independent of $W_0$.
\end{lemma}

\begin{lemma} \label{CrossingAS}
Suppose that $\mbox{law}(\mu_N) \to \mbox{law}(\mu)$ for some random variable $\mu$ in $\mathcal{P}(D[-1,\infty) \times \R)$ with $\mu_N$ and $\mu$ of the same form as in Lemma \ref{LossConvergenceLemma}. Then $\mu$ satisfies property (\ref{crossingPropertyMeasure}) almost surely.
\end{lemma}

\begin{proposition} \label{LimitLaws}
For $N \in \N$, let $(X^N,\Lambda^N)$ be a solution to the particle system
\begin{equation*}
\begin{cases}
    X_t^{i,N} &= \Theta(Y^{i,N})_t = \Theta(X_{0-}^{i,N}+B_t^i - \alpha \Lambda_t^N) = X_{0-}^{i,N}+B_t^i - \alpha \Lambda_t^N + L_t^{i,N} \\
    \tau_{i,N} &= \inf \{t \geq 0: L_t^{i,N} \geq \xi^{i,N} \} \\
    \Lambda_t^n &= \frac{1}{N} \sum_{i=1}^N \mathbbm{1}_{\tau_{i,N} \leq t}
\end{cases}
\end{equation*}
for $\xi^{i,N}, i = 1, \ldots, n$, independent, exponentially distributed random variables with parameter $\kappa>0$. Let $\mu^N$ and $\mu_N^{\Theta}$ denote the corresponding empirical measures defined in (\ref{EmpiricalMeasuresDefn}). Suppose that for some random measure $\mu^{\Theta}$ and some measure $\nu_{0-} \in \mathcal{P}(\R_+)$ we have
\begin{equation*}
    \lim_{N \to \infty} \frac{1}{N} \sum_{i=1}^N \delta_{X_{0-}^{i,N}} = \nu_{0-}
\end{equation*}
and $\mbox{law}(\mu_N^{\Theta}) \to \mbox{law}(\mu^{\Theta})$ along some subsequence. Then $\mu^{\Theta}$ corresponds almost surely to the law of a solution to the McKean-Vlasov problem (\ref{ElasticProblem}) with $\mbox{law}(X_{0-})=\nu_{0-}$.
\end{proposition}

\begin{proof}
We can assume that the empirical measures are given as $\mu_N = \iota_{\alpha}(\zeta_N)$, $\mu_N^{\Theta}= \Theta(\mu_N)$ with $\zeta_N$ as in (\ref{EmpiricalMeasuresDefn}). The same can be assumed for the limit points $\mu= \iota_{\alpha}(\zeta)$ and $\mu^{\Theta}=\Theta(\mu)$ by continuity. The map $t \mapsto \mathbb{E} [\langle \mu, \lambda_t \rangle]$ is increasing. Hence, it has at most countably many points of discontinuity, and the same holds for the map $t \mapsto \mathbb{E} [\int l_t d \zeta(\omega,l,z)]$. We denote the set of discontinuities of these maps by $J$ and fix a $t$ that is a continuity point of all these maps, i.e. $t \notin J$. Then $\langle \mu, \lambda_{t-} \rangle = \langle \mu, \lambda_t  \rangle$ almost surely. From Lemma \ref{CrossingAS} the crossing property holds and we can use Lemma \ref{LossConvergenceLemma} to deduce the almost sure convergence $\lim_{N \to \infty} \langle \mu_N, \lambda_t \rangle = \langle \mu, \lambda_t \rangle$.
The goal is to show that the measure $\mu$ is the law of the process $(Y,\xi)$ of the McKean-Vlasov problem.

\underline{Step 1:} 
We show that for almost every realization $(\omega,z)$, if $\mbox{law}(W,\Lambda,\xi)= \zeta(\omega,z)$, then $\Lambda = \langle \mu(\omega,z),\lambda \rangle$ almost surely. We obtain the estimate
\begin{equation*}
    \mathbb{E} \left[ \int_{\Bar{E}} \Big\vert \lvert l_t - \langle \mu, \lambda_t \rangle \rvert - \lvert l_t - \langle \mu_N, \lambda_t \rangle \rvert \Big\vert d \zeta_N(\omega,l,z) \right] \leq \mathbb{E} \left[\lvert \langle \mu, \lambda_t \rangle - \langle \mu_N, \lambda_t \rangle \rvert \right]
\end{equation*}
and see that the right-hand side vanishes as $N \to \infty$ using the DCT. Since $t \notin J$ the map $l \mapsto l_t$ is continuous on $M$ and it follows that for $\zeta$-almost every $l$ we have
\begin{equation*}
    \mathbb{E} \left[\int_{\hat{E}} \lvert l_t - \langle \mu,\lambda_t \rangle \rvert d \zeta(\omega,l,z) \right] \leq \lim_{N \to \infty} \mathbb{E} \left[\int_{\hat{E}} \lvert l_t - \langle \mu_N,\lambda_t \rangle \rvert d \zeta_N(\omega,l,z) \right] =0
\end{equation*}
because $\langle \mu_N,\lambda \rangle = l$ for $\zeta_N$-almost every $l \in M$, almost surely. Let $t$ range through a countable dense subset of $[0, \infty) \setminus J$ and use right-continuity to conclude.

\underline{Step 2:} 
We show that for a.e. realization $(\omega,z)$ if $\mbox{law}(W,\Lambda,\xi)= \zeta(\omega,z)$, then $\mbox{law}(W_0) = \nu_{0-}$. Let the evaluation map $\hat{\pi}_0$ defined on $\hat{E}$ be given by $(\omega,l,z) \to \omega_0$. Since $\hat{\pi}_0$ is a continuous map, the continuous mapping theorem yields that for the pushforward of $\zeta$ by $\hat{\pi}_0$ we obtain
\begin{equation*}
    \hat{\pi}_0(\zeta) = \lim_{N \to \infty} \hat{\pi}_0(\zeta_N) = \lim_{N \to \infty} \frac{1}{N} \sum_{i=1}^N \delta_{X_{0-}^{i,N}} = \nu_{0-}
\end{equation*}

\underline{Step 3:}
We have $\mbox{law}(X_{0-})=\nu_{0-}$ and $\iota_{\alpha}(\zeta)=\mu$. Step 1 yields that if $\mbox{law}(W,\Lambda,\xi)= \zeta(\omega,z)$ the equation 
\begin{equation*}
    \mbox{law}((W_0,W-W_0)) = \mbox{law}((X_{0-},Y-X_{0-}+ \alpha \langle \mu,\lambda \rangle))
\end{equation*}
holds. 
Using Lemma \ref{BrownianMotionLemma} we obtain that $Y-X_{0-}+ \langle \mu,\lambda \rangle$ is a Brownian motion. 

If we now combine the results from Step 2, Step 3 and apply the pushforward with the Skorokhod map we obtain the theorem by the continuous mapping theorem.
\end{proof}

\begin{theorem}
Suppose that $X_{0-}$ has density $f$ that is bounded and changes monotonicity finitely often on any compact interval. Let $(X^{i,N},\Lambda^N)$ be a physical solution to the particle system (\ref{particleSystem}) and let $\mu^{\Theta}$ be the law of the unique physical solution to the McKean-Vlasov problem (\ref{ElasticProblem}). Then we have the convergence
\begin{equation}
    \lim_{N \to \infty} \frac{1}{N} \sum_{i=1}^N \delta_{X^{i,N}} = \mu^{\Theta}
\end{equation}
weakly on $\mathcal{P}(D[-1,\infty))$.
\end{theorem}

\begin{proof}
Corollary \ref{empiricalMeasureTightness} yields tightness of the empirical-measure process of the particle system (\ref{particleSystem}). Thus, we know that limit points exist. Moreover, Proposition \ref{LimitLaws} shows that limit points are laws of solutions to the McKean-Vlasov problem. If we now restrict to physical solutions to the particle system, we can use the result from \cite{Delarue2015a} that physical solutions converge to physical solutions and combine it with Lemma \ref{ElasticAbsorbingLemma} and the corresponding result for the particle systems. The conditions on the initial density imply uniqueness of physical solutions using Theorem \ref{ThmPhysicalSoltn} and we can conclude weak convergence.  
\end{proof}

\section{Convergence to Absorbing and Reflecting Models}

The elastic boundary is a mixture of absorbing and reflecting boundaries. Thus we can expect to obtain the absorbing and reflecting models as limit cases of the elastic model. In this section we show that as $\kappa \to \infty$ the solution to the elastic model converges to the solution of the absorbing feedback model. On the other hand, if we consider the limit as $\kappa \to 0$, the elastic model converges to a reflecting Brownian motion. 

Our first goal is to show convergence for the loss process $\Lambda^{\kappa}$ as $\kappa \to 0$ and  $\kappa \to \infty$.

\begin{proposition} \label{lossProcessConv}
Let $\Lambda^{\kappa}$ be a physical solution to the McKean-Vlasov problem  (\ref{ElasticProblem}) with parameter $\kappa>0$. Then we have the two limits
\begin{equation}
    \lim_{\kappa \to 0} \Lambda^{\kappa} = 0, \quad \text{in M1} , \quad \text{and} \quad  \lim_{\kappa \to \infty} \Lambda^{\kappa} = \Lambda^{\infty} , \quad \text{in M1}.
\end{equation}
Moreover, $\Lambda^{\infty}$ is a physical solution to the absorbing feedback model.
\end{proposition}

\begin{proof}
Let $\kappa_1, \kappa_2>0$ be such that $\kappa_1 < \kappa_2$ and consider the two related McKean-Vlasov problems, one with $\xi^{\kappa_1}$ and one with $\xi^{\kappa_2}$. Following \cite{Cuchiero2020} we define the operator $\Gamma_{\kappa}: M \to M$ by
\begin{equation}
    \begin{cases}
        X_t^{\kappa,l} = X_{0-} + \xi^{\kappa} + B_t - \alpha l_t, \quad t \geq 0 \\
        \Gamma_{\kappa}[l]_t = \mathbb{P}(\tau^{\kappa,l} \leq t), \quad t \geq 0, \quad \text{where} \quad \tau^{\kappa,l} = \inf\{t \geq 0, X_t^{\kappa,l} \leq 0 \}.
    \end{cases}
\end{equation}
A solution to the absorbing feedback model is a fixed point of the operator $\Gamma_{\kappa}$ and in \cite{Cuchiero2020} the authors show that the minimal solution and therefore also the unique physical solution can be obtained by $\lim_{n \to \infty} \Gamma^{(n)}_{\kappa}[0]$. Moreover, by Lemma \ref{ThmPhysicalSoltn} the solution to this absorbing model is a solution to (\ref{ElasticProblem}). Using the operators $\Gamma_{\kappa_i},i=1,2$ we have
\begin{align*}
    \Gamma_{\kappa_1}[0]_t = \mathbb{P}(\inf_{s \leq t} X_{0-}+\xi^{\kappa_1}+B_s \leq 0) \leq \mathbb{P}(\inf_{s \leq t} X_{0-}+\xi^{\kappa_2}+B_s \leq 0) =\Gamma_{\kappa_2}[0]_t, 
\end{align*}
for all $t \geq 0$. Assume that 
\begin{equation} \label{InductionHyp}
    \Gamma^{(n)}_{\kappa_1}[0]_t \leq \Gamma^{(n)}_{\kappa_2}[0]_t, \quad \text{for all } t \geq 0,
\end{equation}
for some $n \in \N$. Then we have
\begin{align*}
    \Gamma^{(n+1)}_{\kappa_1}[0]_t = \Gamma_{\kappa_1}[\Gamma^{(n)}_{\kappa_1}[0]]_t &= \mathbb{P}( \inf_{s \leq t} X_{0-} + \xi^{\kappa_1}+B_s - \Gamma^{(n)}_{\kappa_1}[0]_s\leq 0) \\
    &\leq  \mathbb{P}( \inf_{s \leq t }X_{0-} + \xi^{\kappa_2}+B_s - \Gamma^{(n)}_{\kappa_2}[0]_s\leq 0) \\
    &= \Gamma_{\kappa_2}[\Gamma^{(n)}_{\kappa_2}[0]]_t =\Gamma^{(n+1)}_{\kappa_2}[0]_t
\end{align*}
for all $t \geq 0$ and by induction the relation (\ref{InductionHyp}) is true for all $n \in \N$. Therefore, we have that, for all $t \geq 0$, $\Lambda_t^{\kappa}$ is increasing in $\kappa$. Moreover it is bounded between zero and one. This means that limit points exist in both cases and we can define
\begin{equation*}
    \Lambda_t^0 = \lim_{\kappa \to 0} \Lambda_t^{\kappa}, \quad \text{for all $t \geq 0$}, \quad \text{and} \quad  \Lambda^{\infty}_t = \lim_{\kappa \to \infty} \Lambda_t^{\kappa}, \quad \text{for all $t \geq 0$}.
\end{equation*}

In the first case we can use the reverse Fatou lemma. This yields
\begin{align*}
    \Lambda_t^0 = \lim_{\kappa \to 0} \Lambda_t^{\kappa} 
    \leq \mathbb{P}(\lim_{\kappa \to 0} X_{0-} + \xi^{\kappa}+B_t - \alpha \Lambda_t^{\kappa} \leq 0) 
    = 0,
\end{align*}
for all $t \geq 0$. Thus, we obtain $\Lambda^0 = 0$. By \cite{Whitt} Corollary 12.5.1 convergence in the M1-topology follows. 

For the second case we have
\begin{align*}
    \Lambda^{\infty}_t = \lim_{\kappa \to \infty} \Lambda_t^{\kappa}= \lim_{\kappa \to \infty} \mathbb{P}(X_{0-} + \xi^{\kappa}+B_t - \alpha \Lambda_t^{\kappa} \leq 0)
    = \mathbb{P}(X_{0-}+ B_t -\alpha \Lambda^{\infty}_t \leq 0)
\end{align*}
where we applied the Portmanteau Theorem in the last step which is possible because the law of the random variable under consideration has no atoms. Hence $\Lambda^{\infty}$ is a solution to the absorbing feedback model with initial condition $X_{0-}$. Another application of \cite{Whitt} Corollary 12.5.1 yields the desired convergence in M1. 

To show that $\Lambda^{\infty}$ is a physical solution we let $\Lambda^a$ be a solution to the absorbing feedback model. Then we have
$\Lambda^{\kappa} \leq \Lambda^a$ for all $\kappa>0$ and by taking limits we obtain $\Lambda^{\infty} \leq \Lambda^a$. Thus $\Lambda^{\infty}$ is the minimal solution and therefore by \cite{Cuchiero2020} a physical solution.
\end{proof}

\begin{proposition} \label{IndicatorProp}
Let $\tau^{\kappa}$ be the elastic stopping time in the feedback model (\ref{ElasticProblem}) with parameter $\kappa>0$ and denote by $\tau^{\infty}$ the absorbing stopping time in the absorbing feedback model. Then we have the following almost sure convergences in the $M1$-topology
\begin{equation}
    \mathbbm{1}_{t<\tau^{\kappa}} \to 1, \quad \text{as } \kappa \to 0; \quad \mathbbm{1}_{t<\tau^{\kappa}} \to \mathbbm{1}_{t<\tau^{\infty}}, \quad \text{as } \kappa \to \infty. 
\end{equation}
\end{proposition}

\begin{proof}
As in Section 4 we can see that the crossing property is satisfied. We can then use Lemma \ref{IndicatorContinuity} to infer convergence for all $t \in [0,\infty)$ in a co-countable set. Since the functions we consider in this proposition are all monotone, the convergence in M1 follows from \cite{Whitt} Corollary 12.5.1. 
\end{proof}

Equipped with these results we can turn to our goal and show convergence for $\mathbbm{1}_{t<\tau^{\kappa}}X^{\kappa}$.

\begin{theorem}
Let $X^{\kappa}$ be the elastic feedback process corresponding to a physical solution to the McKean-Vlasov problem (\ref{ElasticProblem}) with parameter $\kappa>0$ and $\tau^{\kappa}$ the associated elastic stopping time. Then
\begin{enumerate}
    \item 
    as $\kappa \to \infty,$
    $\mathbbm{1}_{t<\tau^{\kappa}}X^{\kappa}$ converges almost surely in the M1-topology to $\mathbbm{1}_{t<\tau^{\infty}}X^{\infty}$ with $X^{\infty}$ given by
    \begin{equation*}
        X^{\infty} = X_{0-} + B - \alpha \Lambda^{\infty}
    \end{equation*}
     with $\Lambda^{\infty}$ being a physical solution to the absorbing feedback model.
    \item 
    as $\kappa \to 0$, $\mathbbm{1}_{t<\tau^{\kappa}}X^{\kappa}$ converges almost surely in the M1-topology and uniformly to $X^0$ given by
    \begin{align*}
        X^0 = \Theta(X_{0-}+B).
    \end{align*}
    The law of $X^0$ is a reflecting Brownian motion started at $X_{0-}$.
\end{enumerate}
\end{theorem}

\begin{proof}
$(i)$
From Proposition \ref{lossProcessConv} we know that $\Lambda^{\kappa}$ converges to $\Lambda^{\infty}$ in $M1$ and that $\Lambda^{\infty}$ is a physical solution. Since the Brownian motion $B$ has continuous paths a.s., the convergence 
\begin{equation*}
    X_{0-}+ B - \alpha \Lambda^{\kappa} \to X^{\infty}, \quad \text{as } \kappa \to \infty
\end{equation*}
in M1 follows from \cite{Whitt} Theorem 12.7.3. By continuity of the Skorokhod map $\Theta$ with respect to the M1-topology we can deduce the convergence
\begin{equation*}
    X^{\kappa} \to \Theta(X^{\infty}), \quad \text{as } \kappa \to \infty
\end{equation*}
in M1. Moreover, by Proposition \ref{IndicatorProp} we have the M1-convergence $\mathbbm{1}_{t<\tau^{\kappa}} \to \mathbbm{1}_{t<\tau^{\infty}}$. Note that both processes $\Theta(X^{\infty})$ and $\mathbbm{1}_{t<\tau^{\infty}}$ only have downward jumps. Thus, we can use \cite{Whitt} Theorem 13.3.2 to obtain the desired convergence
\begin{equation*}
    \mathbbm{1}_{t<\tau^{\kappa}}X^{\kappa} \to \mathbbm{1}_{t<\tau^{\infty}}\Theta(X^{\infty}) , \quad \text{as } \kappa \to \infty,
\end{equation*}
in M1. Since $\tau^{\infty}$ is the first hitting time of zero by $X^{\infty}$ we have the equality
\begin{equation*}
    \mathbbm{1}_{t<\tau^{\infty}}\Theta(X^{\infty}) = \mathbbm{1}_{t<\tau^{\infty}} X^{\infty},
\end{equation*}
which finishes the proof. \\
$(ii)$ Using the results from Proposition \ref{lossProcessConv} and Proposition \ref{IndicatorProp} and following the same arguments as in $(i)$ yields the convergence
\begin{equation}
    \mathbbm{1}_{t<\tau^{\kappa}}X^{\kappa} \to X^0, \quad \text{as } \kappa \to 0,
\end{equation} 
in M1. Since $X^0$ has continuous paths almost surely we obtain the  uniform convergence almost surely by \cite{Whitt} Theorem 12.4.1. The characterization of $X^0$ as a reflecting Brownian motion then follows by applying Tanaka's formula to $X_{0-}+B$ and using \cite{revuz} Chapter VI Corollary 2.2.
\end{proof}

\printbibliography
\noindent
This research has been supported by the EPSRC Centre for Doctoral Training in Mathematics of Random Systems: Analysis, Modelling and Simulation (EP/S023925/1).

\end{document}